\documentclass[preprint,review,10pt]{elsarticle}

\usepackage{amsmath,amssymb,amsfonts,amsthm}
\usepackage{hyperref}
\usepackage{graphicx}
\usepackage{mathrsfs}

\usepackage{float}
\usepackage{color}
\usepackage{cases}

\usepackage{pifont}
\newtheorem{assumption}{Assumption}

\newtheorem{lemma}{Lemma}
\newtheorem{remark}{Remark}
\newtheorem{theorem}{Theorem}

\let\mr=\mathrm
\usepackage{latexsym}
\usepackage{mathrsfs}
\usepackage{dsfont}
\usepackage{eufrak}
\usepackage{multirow}

\usepackage{graphicx}
\usepackage{hyperref}
\usepackage{color}
\usepackage{diagbox}

\usepackage{enumerate}
\usepackage{cases}
\usepackage{xcolor}

\begin{document}

\begin{frontmatter}

\title{Supercloseness of the NIPG method for a  singularly perturbed convection diffusion problem on Shishkin mesh in 2D\tnoteref{funding} }


\author[label1] {Xiaoqi Ma\fnref{cor2}}
\author[label1] { Jin Zhang\corref{cor1}}
\cortext[cor1] {Corresponding address: jinzhangalex@hotmail.com}
\fntext[cor2] {Email: xiaoqiMa@hotmail.com }
\address[label1]{School of Mathematics and Statistics, Shandong Normal University, Jinan 250014, China}

\begin{abstract}
As a popular stabilization technique, the nonsymmetric interior penalty Galerkin (NIPG) method has significant application value in computational fluid dynamics. 
 In this paper, we study the  NIPG method for a typical two-dimensional singularly perturbed convection diffusion problem on a Shishkin mesh. According to the characteristics of the solution, the mesh and numerical scheme, a new composite interpolation is introduced. In fact, this interpolation is composed of a vertices-edges-element interpolation within the layer and a local $L^{2}$-projection outside the layer. On the basis of that, by selecting penalty parameters on different types of interelement edges, we further obtain the supercloseness of almost $k+\frac{1}{2}$ order in an energy norm. Here $k$ is the degree of piecewise polynomials. Numerical tests support  our theoretical conclusion.
\end{abstract}

\end{frontmatter}
\section{Introduction}
We consider a typical two-dimensional singularly perturbed problem:
\begin{equation}\label{eq:S-1}
\begin{aligned}
& Lu:=-\varepsilon \Delta u+\textbf{b}\cdot \nabla u+cu=f, \quad \text{in $ \Omega:= (0,1)\times (0,1)$},\\
& u=0,\quad\text{on $\Gamma=\partial\Omega$},
\end{aligned}
\end{equation}
where the perturbation parameter $\varepsilon$ satisfies $0<\varepsilon \ll 1$, $\textbf{b}$, $c$, $f$ are sufficiently smooth functions and satisfy the following conditions
\begin{equation}\label{eq:SPP-condition-1}
\begin{aligned}
&\textbf{b}(x, y)=(b_{1}(x, y), b_{2}(x, y))\ge (\beta_{1}, \beta_{2})>0,\quad c(x, y)\ge 0, \quad (x, y)\in \overline{\Omega},\\
&c_{0}^{2}(x, y):= \left(c-\frac{1}{2}\nabla\cdot \textbf{b}\right)(x, y)\ge \gamma_{0}>0,\quad (x, y)\in \overline{\Omega},\\
&f(0, 0)=f(1, 0)=f(0, 1)=f(1 ,1)=0,
\end{aligned}
\end{equation}
with some fixed constants $\beta_{1}$, $\beta_{2}$ and $\gamma_{0}$.  Because the parameter $\varepsilon$ can be arbitrarily small, the solution to \eqref{eq:S-1} usually exhibits boundary layers of width $\mathcal{O}(\varepsilon \ln (1/\varepsilon) )$ near $x= 1$ and $y= 1$. This is one of the classical behaviors of singular perturbation problems, namely that the solutions of these problems might change rapidly in a thin region.
However, most traditional numerical methods can not capture this rapid change. To overcome this difficulty, many strategies have been formulated. Among them, as a local refinement strategy, layer-adapted meshes are becoming more and more popular \cite{Lin1:2010-L, Mil1Rio2Shi3:1996-F, Roo1Sty2Tob3:2008-R}. There are two typical layer-adapted meshes in the literature---Bakhvalov meshes and Shishkin meshes (see \cite{Shi1:1992-D}). Due to the simple structure of Shishkin meshes, they have been widely used in the convergence theory of finite element methods; see \cite{Dur1Lom2Pri3:2013-S, Liu1Sty2Zha3:2018-S, Roo1Sty2Tob3:2008-R, Zha1Liu2:2017-S} and its references.

Discontinuous Galerkin finite element methods (DGFEMs) were developed in the early 1970s. 
So far, many variants have been introduced and studied; see \cite{Arn1Bre2Coc3:2002-U, Che1:2021-O, Coc1Don2Guz3:2008-O, Xie1Zhu2Zho3:2013-U, Zhu1Yan2Yin3:2015-H}. They allow the discontinuity of finite element functions on the element interface, and allow high-order schemes to be constructed in a natural way. 
As a variant of DGFEMs, the nonsymmetric interior penalty Galerkin (NIPG) method has its unique advantages.
For example, the NIPG method is increasingly favored by researchers  \cite{Roo1Zar2:2003-motified, Zar1Roo2:2005-I, Zhu1Yan2Yin3:2015-H},  because 
it is  coercive on any mesh.

In order to obtain more accurate numerical solutions, we focus on the supercloseness of the NIPG method. Here 'supercloseness' means that the numerical solution is closer to the interpolation of the exact solution rather than the exact solution itself. For general theory and development of supercloseness, readers can refer to \cite{Dur1Lom2Pri3:2013-S, Li1Whe2:2000-U, Liu1Sty2Zha3:2018-S, Sty1Tob2:2003-motified, Zha1Liu2:2017-S} and their references. In \cite{Roo1Zar2:2007-motified} and \cite{Zar1:2009-C}, the authors used the combination of the standard Galerkin finite element method and the NIPG method to study two different convection diffusion problems in 2D. By means of a nodal bilinear interpolating function in the layer and a local $L^{2}$-projection outside the layer, they obtained the supercloseness of almost $\frac{3}{2}$ order. However, this approach is neither pure nor high-order NIPG method. Compared with this low-order coupling method, on the one hand, the pure NIPG method will be simpler in the implementation of programming, and will be more flexible in the selection of meshes and polynomial degrees on the elements; on the other hand, the high-order NIPG method can achieve higher accuracy. To the best of our knowledge, there is no literature that studies the supercloseness of two-dimensional singularly perturbed convection diffusion problems by using high-order NIPG methods.

The purpose of this paper is to prove the supercloseness of a two-dimensional singularly perturbed problem by using a high-order NIPG method on Shishkin mesh.
In this process, we construct a new composite interpolation, which is composed of a local $L^{2}$-projection outside the layer and a vertices-edges-element interpolation inside the layer. For one thing, we can use the continuity of vertices-edges-element interpolation to eliminate the estimation of some jump terms; and use some convergence properties of this interpolation to improve the accuracy of diffusion term in the layer.  For another, according to the definition of the local $L^{2}$ projection, the convergence order of convective term outside the layer can be improved. 
Actually, these are the two main difficulties that we need to overcome in the analysis of supercloseness. 
 Based on this new interpolation, by defining penalty parameters at different interfaces, we obtain the supercloseness of almost $k+\frac{1}{2}$ order in an energy norm, which appears in the literature for the first time. Finally, a numerical experiment is presented to verify the main conclusion.

The following is the outline of this article. In Section 2, we present \emph{a prior} information of the solution and some basic concepts. Then in Section 3, a piecewise uniform mesh and the corresponding NIPG method are introduced. In Section 4, we design a new composite interpolation and derive some interpolation error estimates. Finally, we provide the main result of supercloseness, which can be verified by a numerical experiment.

In this paper, let $k\ge 1$ be a fixed integer and $C$ represent a generic positive constant which is independent of the mesh parameter $N$ and $\varepsilon$.

\section{\emph{A prior} information and notation}\label{sec:mesh,method}
\subsection{\emph{A prior} information}
Below, we will provide some prior information about the solution, which plays a significant role in the error analysis.
\begin{theorem}\label{eq: SS-1}
The exact solution $u$ of problem \eqref{eq:S-1} can be decomposed as $$u = S + E=S+ E_{1}+ E_{2} +E_{3},$$ where $S$ is the regular part, $E_{1}$ and $E_{2}$ are boundary layers along the sides $x = 1$ and $y = 1$ of $\Omega$ respectively, while $E_{3}$ is a corner layer at $(1, 1)$. Moreover, for all $(x, y)\in \overline{\Omega}$ and $0\le i+j\le k+1$, there is
\begin{equation}\label{eq:decomposition}
\begin{aligned}
&\left|\frac{\partial^{i+j}S}{\partial x^{i}\partial y^{j}}(x, y)\right|\le C,\\
&\left|\frac{\partial^{i+j}E_{1}}{\partial x^{i}\partial y^{j}}(x, y)\right|\le C\varepsilon^{-i}e^{-\beta_{1}(1-x)/\varepsilon},\\
&\left|\frac{\partial^{i+j}E_{2}}{\partial x^{i}\partial y^{j}}(x, y)\right|\le C\varepsilon^{-j}e^{-\beta_{2}(1-y)/\varepsilon},\\
&\left|\frac{\partial^{i+j}E_{3}}{\partial x^{i}\partial y^{j}}(x, y)\right|\le C\varepsilon^{-(i+j)}e^{-\beta_{1}(1-x)/\varepsilon}e^{-\beta_{2}(1-y)/\varepsilon}.
\end{aligned}
\end{equation}
\end{theorem}
\begin{proof}
This conclusion can be obtained by directly referring to \cite{Roo1Sty2Tob3:2008-R}.
\end{proof}
\begin{remark}
The arguments in \cite{Lin1Sty2:2001-A} make this assumption credible if we impose sufficient compatibility conditions on $f$, see \cite{Sty1:2005-S} for more details.

\end{remark}
Before introducing our methods, we first present some necessary notations.
\subsection{Notation}
For an open set $G\subset\Omega$, assume that $L^{2}(G)$ denotes the standard Lebesgue space with the inner product $(\cdot, \cdot)_{G}$ and the norm $\Vert\cdot\Vert_{L^{2}(G)}$. The standard Sobolev spaces are defined by $W^{k}_{p}(G)$ and $H^{k}(G)\equiv W^{k}_{2}(G)$ with the corresponding norm $\Vert\cdot\Vert_{W^{k}_{p}(D)}$ and seminorm $|\cdot|_{W^{k}_{p}}(G)$.

Suppose that $\mathcal{T}$ is a general partitioning of $\Omega$, which consists of disjoint open parallel rectangles $\kappa$ such that $\Omega=\bigcup\limits_{\kappa\in \mathcal{T}}\kappa$. In particular, we shall allow anisotropic (shape-irregular) mesh on a part of $\Omega$, but assume that there are no hanging nodes. 

For each element $\kappa\in\mathcal{T}$, set a nonnegative integer  $s_{\kappa}$ and define a broken Sobolev space of composite order $\textbf{s} = \{s_{\kappa} : \kappa\in\mathcal{T}\}$ with
$$H^{\textbf{s}}(\Omega, \mathcal{T}) = \{ w\in L^{2}(\Omega) : w|_{\kappa}\in H^{s_{\kappa}} (\kappa), \forall\kappa\in \mathcal{T}\}.$$
Then define the corresponding norm and seminorm by
$$\Vert w\Vert_{\textbf{s}, \mathcal{T}} = \left(\sum_{\kappa\in\mathcal{T}} \Vert w\Vert^{2}_{H^{s_{\kappa}}(\kappa)}\right)^{\frac{1}{2}}\quad\text{and}\quad |w|_{\textbf{s}, \mathcal{T}} = \left(\sum_{\kappa\in \mathcal{T}} |w|^{2}_{H^{s_{\kappa}}(\kappa)}\right)^{\frac{1}{2}},$$
respectively. When $s_{\kappa} = s$ for all $\kappa\in \mathcal{T}$, we write $H^{s}(\Omega, \mathcal{T})$, $\Vert w\Vert_{s, \mathcal{T}}$, and $|w|_{s,\mathcal{T}}$. In addition, the broken gradient $\nabla_{\mathcal{T}} w$ of a function $w \in H^{1}(\Omega, \mathcal{T})$ can be denoted as $(\nabla_{\mathcal{T}} w)|_{\kappa} = \nabla(w|_{\kappa}), \kappa\in \mathcal{T}$.

For any element $\kappa$, let $\partial\kappa$ denote the union of all open edges of $\kappa$. Then we use $\textbf{n}_{\kappa}(x, y)$, or simply $\textbf{n}_{\kappa}$ to represent the unit outward normal vector at the point $(x, y)\in\partial\kappa$. On the basis of that, the inflow and outflow parts of $\partial\kappa$ can be defined as
\begin{equation*}
\begin{aligned}
&\partial_{-}\kappa = \{(x, y)\in \partial\kappa : \textbf{b}(x, y)\cdot \textbf{n}_{\kappa}(x, y)<0\},\\
&\partial_{+}\kappa = \{(x, y)\in \partial\kappa : \textbf{b}(x, y)\cdot \textbf{n}_{\kappa}(x, y)\ge 0\},
\end{aligned}
\end{equation*}
respectively.

Let $\mathcal{E}$ be the set of all one-dimensional open edges of $\mathcal{T}$, and $\mathcal{E}_{int}\subset \mathcal{E}$ be the set of all edges $e \in \mathcal{E}$ contained in $\Omega$. 
Then for each $e\in \mathcal{E}_{int}$, there exist indices $i$ and $j$ such that $i > j$, and $\kappa := \kappa_{i}$ and $\kappa := \kappa_{j}$ share the interface $e$. The (element-numbering-dependent) jump and the mean value of a function $v\in H^{1}(\Omega, \mathcal{T})$ across $e$ are defined as
$$[v]_{e} = v|_{\partial\kappa\cap e} - v|_{\partial\kappa'\cap e},\quad \langle v\rangle_{e} = \frac{1}{2}(v|_{\partial\kappa\cap e} + v|_{\partial\kappa'\cap e}),$$
respectively. Note that when $e \subset \Gamma$, we shall set $[v]_{e} =\langle v\rangle_{e} = v$. Then for $e\in \mathcal{E}_{int}$, we define the (element-numbering-dependent) unit normal vector $\nu = \textbf{n}_{\kappa} = -\textbf{n}_{\kappa'}$ pointing from $\kappa$ to $\kappa'$; furthermore, for convenience, if $e \subset \partial\kappa\cap\Gamma$, we also use $\nu$ to denote the unit outward normal vector $\textbf{n}_{\kappa}$ on $\Gamma$. 

For any element $\kappa\in \mathcal{T}$ and $v\in H^{1}(\kappa)$, define the inner trace of $v|_{\kappa}$ on $\partial\kappa$ by $v^{+}_{\kappa}$. When $\partial_{-}\kappa/ \Gamma\neq \varnothing$, for some $\kappa\in \mathcal{T}$, for each $(x, y)\in \partial_{-}\kappa/ \Gamma$ there exists a unique $\kappa'\in \mathcal{T}$ such that $(x, y) \in \partial_{+}\kappa'$. Now for a function $v\in H^{1}(\Omega, \mathcal{T})$ and some $\kappa \in \mathcal{T}$ satisfying $\partial_{-}\kappa/ \Gamma\neq \varnothing$, the outer trace $v_{\kappa}^{-}$ of $v$ on $\partial_{-}\kappa/ \Gamma$ relative to $\kappa$ can be defined as the inner trace $v^{+}_{\kappa'}$ relative to the element $\kappa'$ such that $\partial_{+}\kappa'\cap (\partial_{-}\kappa/\Gamma)\neq \varnothing$. The jump of $v$ across $\partial_{-}\kappa/\Gamma$ is denoted as
$$\left\lfloor v\right\rfloor_{\kappa} = v_{\kappa}^{+} - v_{\kappa}^{-}.$$
In particular, set $\lfloor v\rfloor|_{\Gamma}=v^{+}$.
To simplify the notation, we will omit indices in the terms $[v]_{e}$, $\langle v\rangle_{e}$, and $\lfloor v\rfloor_{\kappa}$. 
\section{Shishkin mesh and NIPG method}
\subsection{Shishkin mesh}
Based on the information of layers presented in Theorem \ref{eq: SS-1}, we use a piecewise uniform mesh---Shishkin mesh. 
First, the set of mesh points  $\Omega_{N}=\{(x_{i}, y_{j})\in \overline
{\Omega} : i, j=0, 1, \cdots, N\}$ is constructed. Set $h_{i,x} := x_{i+1} - x_{i}$ and $h_{j,y} := y_{j+1} - y_{j}$ for all $0\le i, j\le N-1$. And a mesh rectangle is often written as $\kappa_{ij} = (x_{i}, x_{i+1}) \times (y_{j}, y_{j+1})$ for a specific element and as $\kappa$ for a generic mesh rectangle.

In order to distinguish the layer part from the smooth part, we introduce the mesh transition parameters $\lambda_{x}$ and $\lambda_{y}$,
\begin{equation*}
\lambda_{x}=\min\left\{\frac{1}{2}, \frac{\sigma_{x}\varepsilon}{\beta_{1}}\ln N\right\},\quad \lambda_{y}=\min\left\{\frac{1}{2}, \frac{\sigma_{y}\varepsilon}{\beta_{2}}\ln N\right\},
\end{equation*} 
where $\beta_{1}$ and $\beta_{2}$ are the lower bounds from \eqref{eq:SPP-condition-1} for the convective part $\textbf{b}$. For convenience, we take $\sigma_{x}= \sigma_{y}= \sigma$ and set $\sigma\ge k+\frac{3}{2}$. Let $N\in \mathbb{N}$, $N\ge 8$, be divisible by $2$ in both the $x$-- and $y$--directions. 
Then divide the subintervals $[0, 1-\lambda_{x}]$ and $[1-\lambda_{x}, 1]$ into $N/2$ uniform mesh intervals respectively, and the division of $y$--direction is similar. 
Therefore, Shishkin mesh points can be denoted as
\begin{equation*}
x_{i}=
\left\{
\begin{aligned}
&\frac{2(1-\lambda_{x})}{N}i\quad &&\text{for $i=0,1,...,N/2$},\\
& 1-\lambda_{x}+\frac{2\lambda_{x}}{N}(i-\frac{N}{2})&&\text{for $i=N/2+1,...,N$}
\end{aligned}
\right.
\end{equation*}
and
\begin{equation*}
y_{j}=
\left\{
\begin{aligned}
&\frac{2(1-\lambda_{y})}{N}j\quad &&\text{for $j=0,1,...,N/2$},\\
&1- \lambda_{y}+\frac{2\lambda_{y}}{N}(j-\frac{N}{2})&&\text{for $j=N/2+1,...,N$}.
\end{aligned}
\right.
\end{equation*}

Moreover, divide the domain $\Omega$ as $\Omega=\Omega_{11}\cup \Omega_{12}\cup \Omega_{21}\cap \Omega_{22}$, where
\begin{equation*}
\begin{aligned} 
&\Omega_{11}:= (0, 1-\lambda_{x}) \times (0, 1-\lambda_{y}),\quad\quad \Omega_{12} :=(1-\lambda_{x}, 1) \times (0, 1-\lambda_{y}),\\
& \Omega_{21} :=(0, 1-\lambda_{x}) \times (1-\lambda_{y}, 1),\quad\quad\Omega_{22}: =(1-\lambda_{x}, 1) \times (1-\lambda_{y}, 1)
\end{aligned}
\end{equation*}
and we can easily see that $1-\lambda_{x}=x_{N/2}$ and $1-\lambda_{y}=y_{N/2}$.
\vspace{-0.4cm}
\begin{figure}[H]
\begin{center}
\includegraphics[width=0.6\textwidth]{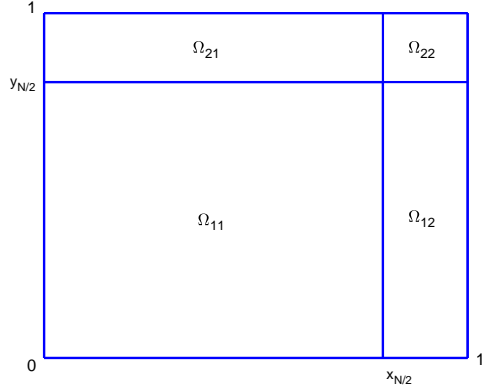}
\vspace{-0.6cm}
\caption{This is partitioning of the domain $\Omega$}
\end{center}
\end{figure}
\vspace{-0.5cm}
In addition, it can be clearly shown that
\begin{equation*}
\begin{aligned}
&\max\{h_{i, x}, h_{j, y}\}\le C\varepsilon N^{-1}\ln N,\quad i, j=N/2, \cdots, N-1,\\
&\max\{h_{i, x}, h_{j, y}\}\le CN^{-1}\quad i, j= 0, 1, \cdots, N/2-1,
\end{aligned}
\end{equation*}
which will be often applied in the error analysis.

\begin{assumption}\label{ass:S-1}
Throughout this paper, assume  that
\begin{equation*}
\varepsilon \le C N^{-1}
\end{equation*}
and in practice, it is not a restriction.
\end{assumption}

\subsection{The NIPG method}
On the above Shishkin mesh, define the finite element space by
$$V_{N}^{k} :=\{v\in L^{2}(\Omega) : v|_{\Gamma} =0, v|_{\kappa} \in \mathcal{Q}_{k}(\kappa)\quad \forall \kappa\in \mathcal{T}\},$$
where $\mathcal{Q}_{k} = span\{x^{i}y^{j} : 1\le i ,j\le k\}$. 
The function $v\in V_{N}^{k}$ is piecewise continuous on each element $\kappa$, that is to say, it allows discontinuity along interelement edges.

Now the weak formulation of \eqref{eq:S-1} is that: Find $u_{h} \in V_{N}^{k}$ such that
\begin{equation}\label{eq:SD}
B(u_{h},v_{h})=L(v_{h}) \quad \forall v_{h} \in V_{N}^{k},
\end{equation}
where $B(u,v)=B_{1}(u,v)+B_{2}(u,v)+B_{3}(u,v)$ and 
\begin{equation*}
\begin{aligned}
B_{1}(u,v)&= \sum_{\kappa\in\mathcal{T}}\varepsilon\int_{\kappa}\nabla u\cdot\nabla v\mr{d}x\mr{d}y-\varepsilon\sum_{e\in\mathcal{E}}\int_{e}\langle\nabla u\cdot\nu\rangle[v]\mr{d}s+\varepsilon\sum_{e\in\mathcal{E}}\int_{e}[u]\langle\nabla v\cdot\nu\rangle\mr{d}s+\sum_{e\in\mathcal{E}}\int_{e}\rho_{e}[u][v]\mr{d}s,
\end{aligned}
\end{equation*}
\begin{equation*}
\begin{aligned}
&B_{2}(u,v)=\sum_{\kappa\in\mathcal{T}}\left(\int_{\kappa}(\textbf{b}\cdot \nabla u)v\mr{d}x\mr{d}y-\int_{\partial_{-}\kappa\cap \Gamma}(\textbf{b}\cdot \textbf{n}_{\kappa})u^{+}v^{+}\mr{d}s-\int_{\partial_{-}\kappa/ \Gamma}(\textbf{b}\cdot \textbf{n}_{\kappa})\lfloor u\rfloor v^{+}\mr{d}s\right),\\
&B_{3}(u,v)=\sum_{\kappa\in\mathcal{T}}\int_{\kappa}cuv\mr{d}x\mr{d}y,\\
&L(v)=\sum_{\kappa\in\mathcal{T}}\int_{\kappa}f v\mr{d}x\mr{d}y,
\end{aligned}
\end{equation*}
where the nonnegative constant $\rho_{e}\ge 0$ for $e \in\mathcal{E}$ are called penalization parameters.
 In the sequel, we shall provide the exact choices of $\rho_{e}$. Before that, let's divide all edges $e \in\mathcal{E}$ into the following types:

The edges of type $\mathcal{M}_{1}$ belong to the set $\Omega^{*}:=[0, 1-\lambda_{x})\times [0, 1-\lambda_{y})$.
They are part of the smooth region and their length is either $h_{i, x}$ or $h_{j, y}$ for $i, j=1, 2, \cdots, N/2-1$. According to the definition of $\lambda_{x}$ and $\lambda_{y}$, one obtains that $h_{i, x}$ or $h_{j, y}$ for $i, j=1, 2, \cdots, N/2-1$ is much larger than $h_{i, x}$ or $h_{j, y}$ for $i, j=N/2, \cdots, N-1$. Therefore we describe these edges as long; the other edges are called short.
Edges of type $\mathcal{M}_{2}$ are also long and belong to the boundary layer region $$(1-\lambda_{x}, 1]\times [0, 1-\lambda_{y})\cup [0, 1-\lambda_{x})\times(1-\lambda_{y}, 1].$$
Since this region also contains short edges, we classify them into the third type. In particular, type $\mathcal{M}_{3}$ contains all short edges in the corner layer $[1-\lambda_{x}, 1]\times [1-\lambda_{y}, 1]$ at the same time. Finally, edges of type $\mathcal{M}_{4}$ are long and lie in $\{1-\lambda_{x}\}\times [0, 1-\lambda_{y}]\cup [0, 1-\lambda_{x}]\times\{1-\lambda_{y}\}$. On that basis, we present the exact values of $\rho_{e}$,
\begin{equation}\label{eq: penalization parameters}
\rho_{e}=\left\{
\begin{aligned}
&1,\quad \text{if e is of type $\mathcal{M}_{1}$},\\
&N^{2},\quad \text{if e is of type $\mathcal{M}_{2}$},\\
&N,\quad \text{if e is of type $\mathcal{M}_{3}$ or $\mathcal{M}_{4}$}.
\end{aligned}
\right.
\end{equation}

\begin{lemma}\label{Galerkin orthogonality property}
Assume that the solution of problem \eqref{eq:S-1} $u \in H^{2}(\Omega, \mathcal{T})$ and $\nabla u\cdot \nu$ is continuous across each interior edge $e$. Then the bilinear form $B(\cdot,\cdot)$ defined in \eqref{eq:SD} has the following Galerkin orthogonality property
\begin{equation*}
B(u-u_{h}, v)=0,\quad \forall v\in V_{N}^{k}.
\end{equation*}
\end{lemma}

According to $B(\cdot,\cdot)$, define the energy norm as
\begin{equation}\label{eq:SS-1}
\begin{aligned}
\Vert v \Vert_{NIPG}^{2}&:=
\sum_{\kappa\in\mathcal{T}}\left(\varepsilon\Vert\nabla v\Vert_{L^{2}(\kappa)}^{2}+\Vert c_{0} v\Vert^{2}_{L^{2}(\kappa)}\right)+\sum_{e\in\mathcal{E}}\int_{e}\rho_{e}[v]^{2}\mr{d}s\\
&+\frac{1}{2}\sum_{\kappa\in\mathcal{T}}\left(\Vert v^{+}\Vert^{2}_{\partial_{-}\kappa\cap\Gamma}+ \Vert v^{+}-v^{-}\Vert^{2}_{\partial_{-}\kappa/\Gamma}+\Vert v^{+}\Vert^{2}_{\partial_{+}\kappa\cap\Gamma}\right) \quad \forall v\in V_{N}^{k}
\end{aligned}
\end{equation}
with the notation
$$(v, w)_{\tau} = \int_{ \kappa} |\textbf{b}\cdot \textbf{n}_{\kappa} |vw\mr{d}s,\quad \tau\subset \partial\kappa,$$
and $\Vert v\Vert^{2}_{\tau} = (v, v)_{\tau}$. 
Then from \eqref{eq:SPP-condition-1} and some direct calculations, one has
\begin{equation}\label{eq:coercity}
B(v_h, v_h) \ge \Vert v_h \Vert_{NIPG}^2\quad \text{for all $v_h\in V_{N}^{k}$},
\end{equation}
which follows that $u_{h}$ is well defined by \eqref{eq:SD}.
%
%
%

\section{Interpolation and interpolation error}
\subsection{Interpolation}
In this section, we design a new interpolation $\Pi u$ of the exact solution $u$, that is 
\begin{equation}\label{eq:H-1}
(\Pi u)|_{\kappa}=\left\{
\begin{aligned}
& (I_{N} u)|_{\kappa}, \quad \text{if $\kappa\subset \overline{\Omega/\Omega_{11}}$},\\
& (P_{h}u)|_{\kappa},\quad \text{if $\kappa\subset \overline{\Omega_{11}}$},
\end{aligned}
\right.
\end{equation}
where $P_{h}u$ is a local $L^{2}$-projection of $u$ in the finite element space, and $I_{N}u$ represents a vertices-edges-element interpolation. Below, we will provide the definitions and corresponding properties of the both interpolations. 

On the one hand, we first define finite elements on the reference cell $\hat{K} = (-1, +1)^{2}$. Let $\hat{a}_{i}$ and $\hat{e}_{i}$ for $i = 1, 2, 3, 4$, denote the vertices and the edges of $\hat{K}$, respectively.  The operator $\hat{I} : C(\overline{\hat{K}})\rightarrow \mathcal{Q}_{k}(\hat{K})$ is defined by requiring \cite{Sty1Tob2:2008-U}
\begin{equation*}
\begin{split}
&(\hat{I} \hat{w})(\hat{a}_{i}) = \hat{w}(\hat{a}_{i}),\quad i = 1, 2, 3, 4,\\ 	
& \int_{\hat{e}_{i}}(\hat{I}\hat{w})q\mr{d}s=\int_{\hat{e}_{i}}\hat{w}q\mr{d}s,\quad i = 1, 2, 3, 4,\quad \forall q\in P_{k-2}(\hat{e}_{i}),\\
&\int_{\hat{K}}(\hat{I}\hat{w})q\mr{d}\xi\mr{d}\eta=\int_{\hat{K}}\hat{w}q\mr{d}\xi\mr{d}\eta,\quad \forall q \in \mathcal{Q}_{k-2}(\hat{K}).
\end{split}
\end{equation*}
Here $P_{k-2}(\hat{e}_{i})$ is the space of polynomials of degree at most $k-2$ in the single variable whose axis is parallel to the edge $\hat{e}_{i}$.

Using the affine transformation which maps from $\hat{K}$ to an arbitrary $\kappa\in \mathcal{T}$, one derives the corresponding approximation operator $I_{K} : C(\overline{K})\rightarrow \mathcal{Q}_{k}(K)$. Then define a continuous global interpolation operator $I_{N} :C(\bar{\Omega}) \rightarrow V_{N}^{k}$ by setting
$$(I_{N} v)|_{\kappa} := I_{K}(v|_{\kappa})\quad \forall \kappa\in \mathcal{T}.$$
In the following we present the corresponding theories and properties of this interpolation according to some arguments in \cite{Sty1Tob2:2008-U}.
\begin{lemma}\label{RR-1}
 Let $\kappa\in \mathcal{T}$. Suppose that $w \in H^{k+2}(\kappa)$ and $I_{N}w \in \mathcal{Q}_{k}(\kappa)$ is its vertices-edges-element interpolation. Then for $v\in \mathcal{Q}_{k}(\kappa)$, 
\begin{equation*}
\begin{split}
&\vert\int_{\kappa}(I_{N}w-w)_{x}v_{x}\mr{d}x\mr{d}y\vert\le Ch_{j, y}^{k+1}\left\Vert \frac{\partial^{k+2}w}{\partial x\partial y^{k+1}}\right\Vert_{L^{2}( \kappa)}\Vert v_{x}\Vert_{L^{2}(\kappa)},\\
&\vert\int_{\kappa}(I_{N}w-w)_{y}v_{y}\mr{d}x\mr{d}y\vert\le Ch_{i, x}^{k+1}\left\Vert \frac{\partial^{k+2}w}{\partial x^{k+1}\partial y}\right\Vert_{L^{2}(\kappa)}\Vert v_{y}\Vert_{L^{2}(\kappa)}
\end{split}
\end{equation*}
\end{lemma}
\begin{proof}
This lemma can be obtained directly according to \citep[Lemma 4]{Sty1Tob2:2008-U}.
\end{proof}
\begin{lemma} \label{eq:interpolation-theory}
 Let $\kappa\in \mathcal{T}$ and there exists a constant $C$ such that the vertices-edges-element approximant $I_{N}v$ satisfies
\begin{equation*}
\begin{split}
&\Vert (v- I_{N}v)_{x}\Vert_{L^{2}(\kappa)}\le C\sum_{l+m=k}h_{i, x}^{l}h_{j, y}^{m}\left\Vert\frac{\partial^{k+1}v}{\partial x^{l+1}\partial y^{m}}\right\Vert_{L^{2}(\kappa)},\\
&\Vert (v- I_{N}v)_{y}\Vert_{L^{2}(\kappa)}\le C\sum_{l+m=k}h_{i, x}^{l}h_{j, y}^{m}\left\Vert\frac{\partial^{k+1}v}{\partial x^{l}\partial y^{m+1}}\right\Vert_{L^{2}(\kappa)}
\end{split}
\end{equation*}
for all $v\in H^{k+1}(\kappa)$.
\end{lemma}
\begin{proof}
From \citep[Lemma 6]{Sty1Tob2:2008-U}, it is straightforward for us to get this lemma.
\end{proof}
\begin{lemma}\label{TTT-1}
Assume $\kappa\in \mathcal{T}$ and $p\in [1, \infty]$, for the the vertices-edges-element interpolation $I_{N}v$ there is
\begin{equation*}\label{WW-3}
\Vert v - I_{N} v\Vert_{L^{p}(\kappa)} \le C \sum_{l+ m=k+1} h^{l}_{i, x}h_{j, y}^{m}\left\Vert\frac{\partial^{k+1}v}{\partial x^{l}\partial y^{m}}\right\Vert_{L^{p}(\kappa)}
\end{equation*}
for all $v\in W^{k+1, p}(\kappa)$.  
\end{lemma} 
\begin{proof}
Actually, this lemma can be derived through \citep[Lemma 7]{Sty1Tob2:2008-U}.
\end{proof}
In addition, applying the same arguments in \cite{Sty1Tob2:2008-U}, we can provide the $L^{\infty}$-stability of the operator $I_{N}$ on each mesh cell $\kappa\in \mathcal{T}$,  $$\Vert I_{N}v\Vert_{L^{\infty}(\kappa)}\le C\Vert v\Vert_{L^{ \infty}(\kappa)},\quad \forall v\in C(\overline{\kappa}).$$

On the other hand, when $k\ge 1$,  define the local $L^{2}$-projection $P_{h}u\in V_{N}^{k}$  restricted to $\kappa\in \mathcal{T}$, $(P_{h}u)|_{\kappa} \in \mathcal{Q}_{k}(\kappa)$ as
\begin{align}
\int_{\kappa}(u - P_{h}u)v \mr{d}x = 0\quad \forall v\in \mathcal{Q}_{k}(\kappa)\label{WW-1},
\end{align}
where $\mathcal{Q}_{k}(\kappa)$ represents a polynomial space whose degree does not exceed $k\times k$. 
\begin{lemma}
\citep[Theorem 1]{Cro1Tho2:1987-motified} Assume that $P_{h}\omega$ is local $L^{2}$-projection   to $\omega$ and $h_{\kappa}=diam(\kappa)$ is the diameter of $\kappa$. Then we have the following approximation property,
\begin{equation}\label{eq:interpolation-theory-1}
\Vert \omega-P_{h}\omega\Vert_{L^{2}(\kappa)}+h_{\kappa}^{\frac{1}{2}}\Vert \omega-P_{h}\omega\Vert_{L^{2}(\partial\kappa)}\le Ch_{\kappa}^{k+1}|\omega|_{H^{k+1}(\kappa)},\quad \kappa\in\mathcal{T}
\end{equation}
for all $\omega\in H^{k+1}(\kappa)$. 
\end{lemma}
On the basis of that, the following lemma can be derived.
\begin{lemma}\label{Interpolation error}
Let Assumption \ref{ass:S-1} hold and $\sigma\ge k+\frac{3}{2}$. Recall $\rho_{e}$ for $e\in \mathcal{E}$ have been presented in \eqref{eq: penalization parameters}. Then on the Shishkin mesh, one has
\begin{align}
&\Vert S-P_{h}S\Vert_{L^{2}(\Omega_{11})}+\Vert u-\Pi u\Vert_{L^{2}(\Omega_{11})}\le CN^{-(k+1)},\label{eq:QQ-1}\\
&\Vert E-I_{N}E\Vert_{L^{2}(\Omega/\Omega_{11})}+\Vert u-\Pi u\Vert_{L^{2}(\Omega/\Omega_{11})}\le C\varepsilon^{\frac{1}{2}}(N^{-1}\ln N)^{k+1},\label{eq:QQ-2}\\
&\Vert \nabla(u-\Pi u)\Vert_{L^{2}(\Omega_{11})}\le CN^{-k}+C\varepsilon^{-\frac{1}{2}}N^{-\sigma}+CN^{1-\sigma},\label{eq:QQ-3}\\
&\Vert \nabla(E_{3}-\Pi E_{3})\Vert_{L^{2}(\Omega_{12}\cup\Omega_{21})}\le CN^{-\sigma}+C\varepsilon^{-\frac{1}{2}}N^{1-\sigma}(\ln N)^{-\frac{1}{2}},\label{eq:QQQ-2}\\
&\Vert \nabla(E_{1}-\Pi E_{1})\Vert_{L^{2}(\Omega_{21})}+\Vert \nabla(E_{2}-\Pi E_{2})\Vert_{L^{2}(\Omega_{12})}\le C\varepsilon^{-\frac{1}{2}}N^{1-\sigma}(\ln N)^{-\frac{1}{2}},\label{eq:QQQ-1}\\
&\Vert u-\Pi u\Vert_{L^{\infty}(\Omega_{11})}\le CN^{-(k+1)},\label{eq:QQ-5}\\
&\Vert u-\Pi u\Vert_{L^{\infty}(\Omega/\Omega_{11})}\le CN^{-(k+1)}(\ln N)^{k+1},\label{eq:QQ-6}\\
&\Vert I_{N}u-u\Vert_{NIPG, \overline{\Omega}_{11}}\le C\varepsilon^{\frac{1}{2}} N^{-k}+CN^{-(k+1)}+C\varepsilon^{\frac{1}{2}} N^{1-\sigma},\label{post-process-1}\\
&\Vert P_{h}u-u\Vert_{NIPG, \overline{\Omega}_{11}}\le CN^{-(k+\frac{1}{2})}(\ln N)^{k+1}.\label{post-process-2}
\end{align}
\end{lemma}
\begin{proof}
According to \eqref{eq:interpolation-theory}, Lemma \ref{eq:interpolation-theory}, Lemma \ref{TTT-1} and \eqref{eq:interpolation-theory-1}, we can obtain (\ref{eq:QQ-1}--\ref{eq:QQ-6}) simply. In the following, we just analyze \eqref{post-process-1} and \eqref{post-process-2}.

From the definition of the NIPG norm \eqref{eq:SS-1} and properties of  vertices-edges-element interpolation on $\Omega$, $[(I_{N}u-u)]_{e}=0, e\subset\overline{\Omega}_{11}$. Therefore, 
\begin{equation*}
\begin{aligned}
\Vert I_{N}u-u\Vert_{NIPG, \overline{\Omega}_{11}}^{2}=\sum_{\kappa\subset\overline{\Omega}_{11}}\left(\varepsilon\Vert\nabla(I_{N}u-u)\Vert^{2}_{L^{2}(\kappa)}+\Vert c_{0}(I_{N}u-u)\Vert^{2}_{L^{2}(\kappa)}\right).
\end{aligned}
\end{equation*}
 Then the triangle inequality, the inverse inequality \citep[Theorem 3.2.6]{Cia1:2002-motified} and \eqref{eq:QQ-3} yield
\begin{equation*}
\begin{aligned}
&|\varepsilon \sum_{\kappa\subset\overline{\Omega}_{11}}\Vert \nabla(I_{N}u-u)\Vert_{L^{2}(\kappa)}^{2}|
\le C\varepsilon N^{-2k}+CN^{-2\sigma}+C\varepsilon N^{2-2\sigma}.
\end{aligned}
\end{equation*}
By using \eqref{eq:QQ-1} and triangle inequalities, one has
\begin{equation*}
|\sum_{\kappa\subset\overline{\Omega}_{11}} \Vert I_{N}u-u\Vert_{L^{2}(\kappa)}^{2}|\le C\sum_{\kappa\subset\overline{\Omega}_{11}}\Vert I_{N}S-S\Vert_{L^{2}(\kappa)}^{2}+C\sum_{\kappa\subset\overline{\Omega}_{11}}\Vert I_{N}E-E\Vert_{L^{2}(\kappa)}^{2}\le CN^{-2(k+1)}.
\end{equation*}
Thus we complete the derivation of \eqref{post-process-1}.

In addition, through the definition of the NIPG norm \eqref{eq:SS-1}, we have
\begin{equation*}
\begin{aligned}
\Vert P_{h}u-u\Vert_{NIPG, \overline{\Omega}_{11}}^{2}&=\sum_{\kappa\subset\overline{\Omega}_{11}}\left(\varepsilon\Vert\nabla(P_{h}u-u)\Vert^{2}_{L^{2}(\kappa)}+\Vert c_{0}(P_{h}u-u)\Vert^{2}_{L^{2}(\kappa)}\right)\\
&+\sum_{e\subset\Omega_{11, int}}\int_{e}\rho_{e}[(P_{h}u-u)]^{2}\mr{d}s+\sum_{e\subset\partial\overline{\Omega}_{11}/\Gamma}\int_{e}\rho_{e}[(P_{h}u-u)]^{2}\mr{d}s\\
&+\sum_{e\subset\partial\overline{\Omega}_{11}\cap\Gamma}\int_{e}\rho_{e}(P_{h}u-u)^{2}\mr{d}s+\frac{1}{2}\sum_{\kappa\subset\overline{\Omega}_{11}}\left(\Vert(P_{h}u-u)^{+}\Vert^{2}_{\partial_{-}\kappa\cap\Gamma}\right.\\
&\left.+\Vert (P_{h}u-u)^{+}-(P_{h}u-u)^{-}\Vert^{2}_{\partial_{-}\kappa/\Gamma}+\Vert (P_{h}u-u)^{+}\Vert^{2}_{\partial_{+}\kappa\cap\Gamma}\right),
\end{aligned}
\end{equation*}
where the set $\Omega_{11, int}$ contains interior edges $e\subset\mathcal{E}_{int}$ that belong to $\Omega_{11}$. According to triangle inequalities, the inverse inequality and \eqref{eq:QQ-3},
\begin{equation*}
\begin{aligned}
&|\varepsilon \sum_{\kappa\subset\overline{\Omega}_{11}}\Vert \nabla(P_{h}u-u)\Vert_{L^{2}(\kappa)}^{2}|
\le C\varepsilon N^{-2k}+CN^{-2\sigma}+C\varepsilon N^{2-2\sigma}.
\end{aligned}
\end{equation*}
From \eqref{eq:QQ-1} and triangle inequalities, one has
\begin{equation*}
|\sum_{\kappa\subset\overline{\Omega}_{11}} \Vert P_{h}u-u\Vert_{L^{2}(\kappa)}^{2}|\le CN^{-2(k+1)}.
\end{equation*}
In particular, recalling $\rho_{e}=1$ for the type $\mathcal{M}_{1}$ and $\rho_{e}=N$ for $e\subset\partial\overline{\Omega}_{11}/\Gamma$, 
\begin{equation*}
\begin{aligned}
&\sum_{e\subset\Omega_{11, int}}\int_{e}\rho_{e}[(P_{h}u-u)]^{2}\mr{d}s+\sum_{e\subset\partial\overline{\Omega}_{11}/\Gamma}\int_{e}\rho_{e}[(P_{h}u-u)]^{2}\mr{d}s+\sum_{e\subset\partial\overline{\Omega}_{11}\cap\Gamma}\int_{e}\rho_{e}(P_{h}u-u)^{2}\mr{d}s\\
&\le C\Vert P_{h}u-u\Vert^{2}_{L^{\infty}(\Omega_{11})}N^{2}N^{-1}+CN\Vert P_{h}u-u\Vert^{2}_{L^{\infty}(\Omega_{12}\cup\Omega_{21})}+C\Vert P_{h}u-u\Vert^{2}_{L^{\infty}(\Omega_{11})}\\
&\le CN^{-(2k+1)}(\ln N)^{2(k+1)}.
\end{aligned}
\end{equation*}
Below let us analyze $\sum_{\kappa\subset\overline{\Omega}_{11}}\left(\Vert(P_{h}u-u)^{+}\Vert^{2}_{\partial_{-}\kappa\cap\Gamma}+\Vert (P_{h}u-u)^{+}\Vert^{2}_{\partial_{+}\kappa\cap\Gamma}\right)$, i.e.,
\begin{equation*}
\begin{aligned}
&\sum_{\kappa\subset\overline{\Omega}_{11}}\left(\Vert(P_{h}u-u)^{+}\Vert^{2}_{\partial_{-}\kappa\cap\Gamma}+\Vert (P_{h}u-u)^{+}\Vert^{2}_{\partial_{+}\kappa\cap\Gamma}\right)\\
&\le C\Vert P_{h}u-u\Vert^{2}_{L^{\infty}(\Omega_{11})}N N^{-1}\\
&\le CN^{-2(k+1)}.
\end{aligned}
\end{equation*}
Using the same method, one has
\begin{equation*}
\begin{aligned}
\sum_{\kappa\subset\overline{\Omega}_{11}}\Vert (P_{h}u-u)^{+}-(P_{h}u-u)^{-}\Vert^{2}_{\partial_{-}\kappa/\Gamma}&\le C\Vert P_{h}u-u\Vert^{2}_{L^{\infty}(\Omega_{11})}N^{2}N^{-1}\le CN^{-(2k+1)}.
\end{aligned}
\end{equation*}
Thus we complete the derivation of \eqref{post-process-2}.
\end{proof}
\begin{theorem}\label{the:main result2}
Let Assumption \ref{ass:S-1} hold. Suppose that $\rho_{e}$ for $e\in \mathcal{E}$ are defined in \eqref{eq: penalization parameters}.
Then on the Shishkin mesh with $\sigma\ge k+\frac{3}{2}$ we derive
\begin{align*}
\Vert u-\Pi u\Vert_{NIPG}\le CN^{-k}(\ln N)^{k},
\end{align*}
where $\Pi u$ is the interpolation of the exact solution $u$ of \eqref{eq:S-1}.
\end{theorem}
\begin{proof}
According to the definition of NIPG norm and Lemma \ref{Interpolation error}, we can easily derive this theorem.
\end{proof}

To simplify the following analysis, we set $\eta: =\Pi u-u$.
\begin{lemma}\label{Ppp-1}
Assume that $u$ is the exact solution of \eqref{eq:S-1} with the decomposition $u = S + E= S+ E_{1}+ E_{2}+ E_{3}$ as in Theorem \ref{eq:S-1}. Then
\begin{equation*}
\left\Vert\frac{\partial\eta}{\partial x}\right\Vert^{2}_{L^{\infty}(\kappa)}\le\left\{
\begin{aligned}
&C\varepsilon^{-2}N^{-2\sigma}+CN^{-2k},\quad \kappa\subset \Omega_{11}\\
&C\varepsilon^{-2}N^{-2\sigma}+CN^{-2k}(\ln N)^{2k},\quad\kappa\subset\Omega_{21}\\
&C\varepsilon^{-2}N^{-2k}(\ln N)^{2k},\quad\kappa\subset\Omega_{12}\cup\Omega_{22}
\end{aligned}
\right.
\end{equation*}
and
\begin{equation*}
\left\Vert\frac{\partial\eta}{\partial y}\right\Vert^{2}_{L^{\infty}(\kappa)}\le\left\{
\begin{aligned}
&C\varepsilon^{-2}N^{-2\sigma}+CN^{-2k},\quad\kappa\subset\Omega_{11}\\
&C\varepsilon^{-2}N^{-2\sigma}+CN^{-2k}(\ln N)^{2k},\quad\kappa\subset\Omega_{12}\\
&C\varepsilon^{-2}N^{-2k}(\ln N)^{2k},\quad\kappa\subset\Omega_{21}\cup\Omega_{22}.
\end{aligned}
\right.
\end{equation*}
\end{lemma}
\begin{proof}
From the decomposition of $u$, Lemma \ref{eq:interpolation-theory} and \eqref{eq:interpolation-theory-1}, we draw this conclusion without any difficulties.
\end{proof}

\section{Supercloseness}
Next let us introduce $\xi:=\Pi u-u_{N}$ and recall $\eta: =\Pi u-u$.
Then from \eqref{eq:coercity} and the Galerkin orthogonality, 
\begin{equation}\label{eq:uniform-convergence-1}
\begin{split}
\Vert \xi \Vert_{NIPG}^2 &\le B(\xi,\xi)=B(\Pi u-u+u-u_{N},\xi) =B(\eta,\xi)\\
&= \sum_{\kappa\in\mathcal{T}}\varepsilon\int_{\kappa}\nabla \eta\cdot\nabla \xi\mr{d}x\mr{d}y+ \sum_{\kappa\in\mathcal{T}}\int_{\kappa}c\eta\xi\mr{d}x\mr{d}y\\
\quad&-\varepsilon\sum_{e\in\mathcal{E}}\int_{e}\langle\nabla \eta\cdot\nu\rangle[\xi]\mr{d}s+\varepsilon\sum_{e\in\mathcal{E}}\int_{e}[\eta]\langle\nabla \xi\cdot\nu\rangle\mr{d}s+\sum_{e\in\mathcal{E}}\int_{e}\rho_{e}[\eta][\xi]\mr{d}s\\
\quad&+\sum_{\kappa\in\mathcal{T}}\left(\int_{\kappa}(\textbf{b}\cdot \nabla \eta)\xi\mr{d}x\mr{d}y-\int_{\partial_{-}\kappa\cap \Gamma}(\textbf{b}\cdot \textbf{n}_{\kappa})\eta^{+}\xi^{+}\mr{d}s-\int_{\partial_{-}\kappa/ \Gamma}(\textbf{b}\cdot \textbf{n}_{\kappa})\lfloor \eta\rfloor \xi^{+}\mr{d}s\right)\\
\quad& =:\mr{I}+\mr{II}+\mr{III}+\mr{IV}+\mr{V}+\mr{VI}.
\end{split}
\end{equation}

In the following, we will analyze the terms on the right-hand side of \eqref{eq:uniform-convergence-1}. Firstly, decompose $\mr{I}$ into
\begin{equation*}
\begin{aligned}
\mr{I}&=\sum_{\kappa\subset \Omega_{11}\cup \Omega_{22}}\varepsilon\int_{\kappa}\nabla\eta\nabla\xi\mr{d}x\mr{d}y+\sum_{\kappa\subset \Omega_{12}\cup\Omega_{21}}\varepsilon\int_{\kappa}\nabla (S-I_{N}S)\nabla\xi\mr{d}x\mr{d}y\\
&+ \sum_{\kappa\subset \Omega_{12}\cup\Omega_{21}}\varepsilon\int_{\kappa}\nabla (E_{3}-I_{N}E_{3})\nabla\xi\mr{d}x\mr{d}y\\
&+\sum_{\kappa\subset\Omega_{12}}\varepsilon\int_{\kappa}\nabla (E_{2}-I_{N}E_{2})\nabla\xi\mr{d}x\mr{d}y+\sum_{\kappa\subset\Omega_{12}}\varepsilon\int_{\kappa}\nabla (E_{1}-I_{N}E_{1})\nabla\xi\mr{d}x\mr{d}y\\
&+\sum_{\kappa\subset\Omega_{21}}\varepsilon\int_{\kappa}\nabla (E_{1}-I_{N}E_{1})\nabla\xi\mr{d}x\mr{d}y+ \sum_{\kappa\subset\Omega_{21}}\varepsilon\int_{\kappa}\nabla (E_{2}-I_{N}E_{2})\nabla\xi\mr{d}x\mr{d}y\\
&=\mathcal{G}_1+\mathcal{G}_2+\mathcal{G}_3+\mathcal{G}_4+\mathcal{G}_5+\mathcal{G}_6+\mathcal{G}_7.
\end{aligned}
\end{equation*}
For $\kappa\subset\Omega_{11}\cup\Omega_{22}$, from H\"{o}lder inequalities and \eqref{eq:QQ-3}, we obtain 
\begin{equation}\label{eq:convergence-1-1}
\begin{aligned}
|\sum_{\kappa\subset\Omega_{11}\cup\Omega_{22}}\varepsilon \int_{\kappa}\nabla\eta\nabla\xi\mr{d}x\mr{d}y|&\le C\left(\sum_{\kappa\subset\Omega_{11}\cup\Omega_{22}}\varepsilon\Vert\nabla\eta\Vert^{2}_{L^{2}(\kappa)}\right)^{\frac{1}{2}}\left(\sum_{\kappa\subset\Omega_{11}\cup\Omega_{22}}\varepsilon\Vert\nabla\xi\Vert^{2}_{L^{2}(\kappa)}\right)^{\frac{1}{2}}\\
&\le C\left(\varepsilon^{\frac{1}{2}}N^{-k}+N^{-\sigma}+\varepsilon^{\frac{1}{2}}N^{1-\sigma}+C\varepsilon^{\frac{1}{2}}N^{-k}(\ln N)^{k+\frac{1}{2}}\right)\Vert\xi\Vert_{NIPG}\\
&\le C\left(N^{-\sigma}+\varepsilon^{\frac{1}{2}}N^{-k}(\ln N)^{k+\frac{1}{2}}\right)\Vert\xi\Vert_{NIPG}.
\end{aligned}
\end{equation}
Furthermore, for $\mathcal{G}_2$ and $\mathcal{G}_3$, using H\"{o}lder inequalities and \eqref{eq:interpolation-theory}, some direct calculations show that
\begin{equation}\label{eq:convergence-1-2}
\begin{aligned}
|\mathcal{G}_2|&\le C\left(\sum_{\kappa\subset\Omega_{12}\cup\Omega_{21}}\varepsilon\Vert\nabla(S- I_{N}S)\Vert^{2}_{L^{2}(\kappa)}\right)^{\frac{1}{2}}\left(\sum_{\kappa\subset\Omega_{12}\cup\Omega_{21}}\varepsilon\Vert\nabla\xi\Vert^{2}_{L^{2}(\kappa)}\right)^{\frac{1}{2}}\\
&\le C\varepsilon^{\frac{1}{2}}N^{-k}\Vert\xi\Vert_{NIPG}
\end{aligned}
\end{equation}
and from \eqref{eq:QQQ-2} there is
\begin{equation}\label{eq:convergence-6}
\begin{aligned}
|\mathcal{G}_3|&\le C\left(\sum_{\kappa\subset\Omega_{12}\cup\Omega_{21}}\varepsilon\Vert\nabla (E_{3}-I_{N}E_{3})\Vert^{2}_{L^{2}(\kappa)}\right)^{\frac{1}{2}}\left(\sum_{\kappa\subset\Omega_{12}\cup\Omega_{21}}\varepsilon\Vert\nabla\xi\Vert^{2}_{L^{2}(\kappa)}\right)^{\frac{1}{2}}\\
&\le C\left(\varepsilon^{\frac{1}{2}}N^{-\sigma}+N^{1-\sigma}(\ln N)^{-\frac{1}{2}}\right)\Vert\xi\Vert_{NIPG}.
\end{aligned}
\end{equation}
Similarly, for $\mathcal{G}_{4}$ and $\mathcal{G}_{6}$, we can obtain that
\begin{equation}\label{RR-2}
\begin{aligned}
&|\sum_{\kappa\subset \Omega_{12}}\varepsilon\int_{\kappa}\nabla (E_{2}-I_{N}E_{2})\nabla\xi\mr{d}x\mr{d}y|+|\sum_{\kappa\subset \Omega_{21}}\varepsilon\int_{\kappa}\nabla (E_{1}-I_{N}E_{1})\nabla\xi\mr{d}x\mr{d}y|\\
&\le C\left[\left(\sum_{\kappa\subset \Omega_{12}}\varepsilon\Vert\nabla(E_{2}-I_{N}E_{2})\Vert^{2}_{L^{2}(\kappa)}\right)^{\frac{1}{2}}+\left(\sum_{\kappa\subset \Omega_{21}}\varepsilon\Vert\nabla(E_{1}-I_{N}E_{1})\Vert^{2}_{L^{2}(\kappa)}\right)^{\frac{1}{2}}\right]\Vert\xi\Vert_{NIPG}\\
&\le CN^{1-\sigma}(\ln N)^{-\frac{1}{2}}\Vert\xi\Vert_{NIPG},
\end{aligned}
\end{equation}
where \eqref{eq:QQQ-1} has been applied. In particular, through Lemma \ref{RR-1} and Theorem \ref{eq: SS-1}, the following estimate can be derived,
\begin{equation}\label{eq:convergence-1-3}
\begin{aligned}
&\vert\sum_{\kappa\subset\Omega_{12}}\varepsilon\int_{\kappa}\nabla (E_{1}-I_{N}E_{1})\nabla\xi\mr{d}x\mr{d}y\vert\\
&\le C\sum_{\kappa\subset\Omega_{12}}\varepsilon\int_{\kappa}(E_{1}-I_{N}E_{1})_{x}\frac{\partial\xi}{\partial x}\mr{d}x\mr{d}y+ C\sum_{\kappa\subset\Omega_{12}}\varepsilon\int_{\kappa}(E_{1}-I_{N}E_{1})_{y}\frac{\partial\xi}{\partial y}\mr{d}x\mr{d}y\\
&\le C\sum_{\kappa\subset\Omega_{12}}\varepsilon h_{j, y}^{k+1}\left\Vert \frac{\partial^{k+2}E_{1}}{\partial x\partial y^{k+1}}\right\Vert_{L^{2}(\kappa)}\left\Vert\frac{\partial\xi}{\partial x}\right\Vert_{L^{2}(\kappa)}+ C\sum_{\kappa\subset\Omega_{12}}\varepsilon h_{i, x}^{k+1}\left\Vert\frac{\partial^{k+2}E_{1}}{\partial x^{k+1}\partial y}\right\Vert_{L^{2}(\kappa)}\left\Vert\frac{\partial \xi}{\partial y}\right\Vert_{L^{2}(\kappa)}\\
&\le C\sum_{\kappa\subset\Omega_{12}}N^{-(k+1)}\Vert e^{-\beta_{1}(1-x)/\varepsilon}\Vert_{L^{2}(\kappa)}\left\Vert\frac{\partial\xi}{\partial x}\right\Vert_{L^{2}(\kappa)}+  C\sum_{\kappa\subset\Omega_{12}} \varepsilon (N^{-1}\ln N)^{k+1}\Vert e^{-\beta_{1}(1-x)/\varepsilon}\Vert_{L^{2}(\kappa)}\left\Vert\frac{\partial\xi}{\partial y}\right\Vert_{L^{2}(\kappa)}\\
&\le CN^{-(k+1)}\left(\sum_{\kappa\subset\Omega_{12}}\varepsilon\left\Vert\frac{\partial \xi}{\partial x}\right\Vert^{2}_{L^{2}(\kappa)}\right)^{\frac{1}{2}}+\varepsilon (N^{-1}\ln N)^{k+1}\left(\sum_{\kappa\subset\Omega_{12}}\varepsilon\left\Vert\frac{\partial \xi}{\partial y}\right\Vert^{2}_{L^{2}(\kappa)}\right)^{\frac{1}{2}}\\
&\le C\left(N^{-(k+1)}+\varepsilon (N^{-1}\ln N)^{k+1}\right)\Vert\xi\Vert_{NIPG}.
\end{aligned}
\end{equation}

Using the similar method, it is straightforward to see that
\begin{equation}\label{RR-3}
\begin{aligned}
\vert\sum_{\kappa\subset\Omega_{21}}\varepsilon\int_{\kappa}\nabla (E_{2}-I_{N}E_{2})\nabla\xi\mr{d}x\mr{d}y\vert
\le C\left(N^{-(k+1)}+\varepsilon (N^{-1}\ln N)^{k+1}\right)\Vert\xi\Vert_{NIPG}.
\end{aligned}
\end{equation}
Combining \eqref{eq:convergence-1-1}, \eqref{eq:convergence-1-2}, \eqref{eq:convergence-6}, \eqref{RR-2}, \eqref{eq:convergence-1-3} and \eqref{RR-3}, one derives
\begin{equation}\label{eq:convergence-1}
\mr{I}\le C\left(\varepsilon^{\frac{1}{2}}N^{-k}(\ln N)^{k+\frac{1}{2}}+N^{-(k+1)}+N^{1-\sigma}(\ln N)^{-\frac{1}{2}}\right)\Vert\xi\Vert_{NIPG}.
\end{equation}

For $\mr{II}$, according to \eqref{eq:QQ-1} and \eqref{eq:QQ-2}, we have
\begin{equation}\label{BB-1}
\begin{aligned}
\vert \sum_{\kappa\in\mathcal{T}}\int_{\kappa}c\eta\xi\mr{d}x\mr{d}y\vert&\le C\left(\sum_{\kappa\in\mathcal{T}}\Vert\eta\Vert^{2}_{L^{2}(\kappa)}\right)^{\frac{1}{2}}\left(\sum_{\kappa\in \mathcal{T}}\Vert\xi\Vert^{2}_{L^{2}(\kappa)}\right)^{\frac{1}{2}}\\
&\le C\left(N^{-(k+1)}+\varepsilon^{\frac{1}{2}}N^{-(k+1)}(\ln N)^{k+1}\right)\Vert\xi\Vert_{NIPG}.
\end{aligned}
\end{equation}

Then, we will analyze the estimate of $\mr{III}$. Before that, let us first introduce an element numbering in $\mathcal{T}$ in the following way: the first element in the mesh is on the bottom-left position and the numbering is from the bottom to the top going from the left to the right, as depicted in Figure \ref{Fig1.eps}.\\
\vspace{-0.6cm}
\begin{figure}[H]
\begin{center}
\includegraphics[width=0.8\textwidth]{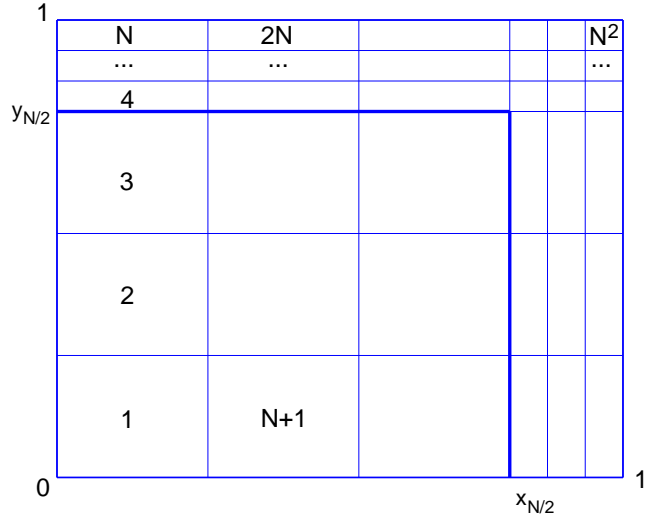}
\vspace{-0.9cm}
\caption{The element numbering in the partitioning $\mathcal{T}$}\label{Fig1.eps}
\end{center}
\end{figure}
\vspace{-0.5cm}
According to the Cauchy-Schwarz inequality, there is
\begin{equation*}
\begin{aligned}
&|\sum_{e\in\mathcal{E}}\int_{e}\varepsilon\langle\nabla\eta\cdot\nu\rangle[\xi]\mr{d}s|\\
&\le C\left(\sum_{e\in\mathcal{E}}\int_{e}\frac{\varepsilon^{2}}{\rho_{e}}\langle\nabla\eta\cdot\nu\rangle^{2}\mr{d}s\right)^{\frac{1}{2}}\left(\sum_{e\in\mathcal{E}}\int_{e}\rho_{e}[\xi]^{2}\mr{d}s\right)^{\frac{1}{2}}\\
&\le C\left(\sum_{e\in\mathcal{E}}\int_{e}\frac{\varepsilon^{2}}{\rho_{e}}\langle\nabla\eta\cdot\nu\rangle^{2}\mr{d}s\right)^{\frac{1}{2}}\Vert\xi\Vert_{NIPG},
\end{aligned}
\end{equation*} 
where for $e\in \mathcal{E}_{int}$, the vector $\nu$ is $\nu = (0, -1)$ for the horizontal edges and $\nu = (-1, 0)$ for the vertical edges, while for $e\in \mathcal{E}\cap\Gamma$, $\nu$ denotes the unit outward normal vector $\textbf{n}_{\kappa}$ on $\Gamma$. Through the definition of $\langle\cdot\rangle$, Lemma \ref{Ppp-1} and  the values of $\rho_{e}$ \eqref{eq: penalization parameters}, we analyze $\mr{III}$ in the following four cases.
\begin{itemize}
\item If e is type $\mathcal{M}_{1}$, recalling $\rho_{e}=1$ and we have
\begin{equation*}
\begin{aligned}
\sum_{e\in\mathcal{E}}\int_{e}\frac{\varepsilon^{2}}{\rho_{e}}\langle\nabla\eta\cdot\nu\rangle^{2}\mr{d}s
&\le C\left(\left\Vert\frac{\partial\eta}{\partial x}\right\Vert^{2}_{L^{\infty}(\Omega_{11})}+\left\Vert\frac{\partial\eta}{\partial y}\right\Vert^{2}_{L^{\infty}(\Omega_{11})}\right)\frac{\varepsilon^{2}}{\rho_{e}}\sum_{e\in\mathcal{E}}\int_{e}\mr{d}s\\
&\le C\left(\varepsilon^{-2}N^{-2\sigma}+N^{-2k}\right)\varepsilon^{2}N^{2}N^{-1}\\
&\le CN^{1-2\sigma}+C\varepsilon^{2}N^{1-2k}.
\end{aligned}
\end{equation*}

\item If e is type $\mathcal{M}_{2}$, we have $\rho_{e}=N^{2}$ and
\begin{equation*}
\begin{aligned}
\sum_{e\in\mathcal{E}}\int_{e}\frac{\varepsilon^{2}}{\rho_{e}}\langle\nabla\eta\cdot\nu\rangle^{2}\mr{d}s
&\le C\left(\left\Vert\frac{\partial\eta}{\partial y}\right\Vert^{2}_{L^{\infty}(\Omega_{21})}+\left\Vert\frac{\partial\eta}{\partial x}\right\Vert^{2}_{L^{\infty}(\Omega_{12})}\right)\frac{\varepsilon^{2}}{\rho_{e}}\sum_{e\in\mathcal{E}}\int_{e}\mr{d}s\\
&\le C\varepsilon^{-2}N^{-2k}(\ln N)^{2k}\frac{\varepsilon^{2}}{N^{2}}N^{2}N^{-1}\\
&\le CN^{-(2k+1)}(\ln N)^{2k}.
\end{aligned}
\end{equation*}

\item If e is of type $\mathcal{M}_{3}$, there is $\rho_{e}=N$ and then
\begin{equation*}
\begin{aligned}
\sum_{e\in\mathcal{E}}\int_{e}\frac{\varepsilon^{2}}{\rho_{e}}\langle\nabla\eta\cdot\nu\rangle^{2}\mr{d}s
&\le C\left(\left\Vert\frac{\partial\eta}{\partial y}\right\Vert^{2}_{L^{\infty}(\Omega_{12}\cup\Omega_{22})}+\left\Vert\frac{\partial\eta}{\partial x}\right\Vert^{2}_{L^{\infty}(\Omega_{21}\cup\Omega_{22})}\right)\frac{\varepsilon^{2}}{\rho_{e}}\sum_{e\in\mathcal{E}}\int_{e}\mr{d}s\\
&\le C\left(\varepsilon^{-2}N^{-2\sigma}+N^{-2k}(\ln N)^{2k}+\varepsilon^{-2}N^{-2k}(\ln N)^{2k}\right)\frac{\varepsilon^{2}}{N}\varepsilon N\ln N\\
&\le C\varepsilon N^{-2k}(\ln N)^{2k+1}.
\end{aligned}
\end{equation*}

\item If e is of type $\mathcal{M}_{4}$, recalling $\rho_{e}=N$ and there is
\begin{equation*}
\begin{aligned}
\sum_{e\in\mathcal{E}}\int_{e}\frac{\varepsilon^{2}}{\rho_{e}}\langle\nabla\eta\cdot\nu\rangle^{2}\mr{d}s
&\le C\left(\left\Vert\frac{\partial\eta}{\partial y}\right\Vert^{2}_{L^{\infty}(\Omega_{11}\cup\Omega_{21})}+\left\Vert\frac{\partial\eta}{\partial x}\right\Vert^{2}_{L^{\infty}(\Omega_{11}\cup\Omega_{12})}\right)\frac{\varepsilon^{2}}{\rho_{e}}\sum_{e\in\mathcal{E}}\int_{e}\mr{d}s\\
&\le C\left(\varepsilon^{-2}N^{-2\sigma}+N^{-2k}+\varepsilon^{-2}N^{-2k}(\ln N)^{2k}\right)\frac{\varepsilon^{2}}{N}N N^{-1}\\
&\le CN^{-(2k+1)}(\ln N)^{2k}.
\end{aligned}
\end{equation*}
\end{itemize}
So far, combined with the above derivation, it is straightforward to get
\begin{equation}\label{DDD-2}
\mr{III}\le C\left(N^{-(k+\frac{1}{2})}(\ln N)^{k}+\varepsilon^{\frac{1}{2}}N^{-k}(\ln N)^{k+\frac{1}{2}}\right)\Vert\xi\Vert_{NIPG}.
\end{equation}

For $\mr{IV}$, we divide it into two parts: on $\overline{\Omega}_{11}$ and $\overline{\Omega/\Omega_{11}}$.
Then according to \eqref{eq:H-1} and the properties of vertices-edges-element interpolation, we just analyze this term on $\overline{\Omega}_{11}$. For $e\subset \Omega^{*}:=[0, 1-\lambda_{x})\times[0, 1-\lambda_{y})$, from the Cauchy–Schwarz inequality and inverse inequality
$$\int_{e}\langle\nabla\xi\cdot\nu\rangle^{2}\mr{d}s\le h_{e}^{-1}\int_{\kappa}(\nabla\xi)^{2}\mr{d}x\mr{d}y,\quad \xi\in V_{N}^{k},$$
 one has
\begin{equation*}
\begin{aligned}
&|\sum_{e\subset \Omega^{*}}\varepsilon\int_{e}[\eta]\langle \nabla\xi\cdot\nu\rangle\mr{d}s|\le C\sum_{e\subset \Omega^{*}}\varepsilon\Vert\eta\Vert_{L^{\infty}(\Omega_{11})}\int_{e}\langle\nabla\xi\cdot\nu\rangle\mr{d}s\\
&\le C\varepsilon N^{-(k+1)}\left(\sum_{e\subset \Omega^{*}}\int_{e}1^{2}\mr{d}s\right)^{\frac{1}{2}}\left(\sum_{e\subset \Omega^{*}}\int_{e}\langle\nabla\xi\cdot\nu\rangle^{2}\mr{d}s\right)^{\frac{1}{2}}\\
&\le C\varepsilon N^{-(k+1)}N^{\frac{1}{2}}\left(\sum_{e\subset \Omega^{*}}\int_{e}\langle\nabla\xi\cdot\nu\rangle^{2}\mr{d}s\right)^{\frac{1}{2}}\\
&\le C\varepsilon N^{-(k+1)}N^{\frac{1}{2}}h_{e}^{-\frac{1}{2}}\left(\sum_{\kappa\subset\Omega^{*}}\int_{\kappa}(\nabla\xi)^{2}\mr{d}x\mr{d}y\right)^{\frac{1}{2}}\\
&\le C\varepsilon^{\frac{1}{2}} N^{-(k+\frac{1}{2})}N^{\frac{1}{2}}\left(\sum_{\kappa\subset\Omega^{*}}\varepsilon\int_{\kappa}(\nabla\xi)^{2}\mr{d}x\mr{d}y\right)^{\frac{1}{2}}\\
&\le C\varepsilon^{\frac{1}{2}} N^{-k}\Vert\xi\Vert_{NIPG},
\end{aligned}
\end{equation*}
here  $h_{e}$ represents the minimum length of all edges $e$ in $\kappa\subset\Omega_{11}$.  
Furthermore, for $e\subset\partial\Omega_{11}/\Gamma$, using the similar method and \eqref{eq:QQ-6}, 
\begin{equation*}
\begin{aligned}
&|\sum_{e\subset \partial\Omega_{11}/\Gamma}\varepsilon\int_{e}[\eta]\langle \nabla\xi\cdot\nu\rangle\mr{d}s|\le C\sum_{e\subset \partial\Omega_{11}/\Gamma}\varepsilon\Vert\eta\Vert_{L^{\infty}(\Omega/\Omega_{22})}\int_{e}\langle\nabla\xi\cdot\nu\rangle\mr{d}s\\
&\le C\varepsilon N^{-(k+1)}(\ln N)^{k+1}\left(\sum_{e\subset \partial\Omega_{11}/\Gamma}\int_{e}1^{2}\mr{d}s\right)^{\frac{1}{2}}\left(\sum_{e\subset\partial\Omega_{11}/\Gamma}\int_{e}\langle\nabla\xi\cdot\nu\rangle^{2}\mr{d}s\right)^{\frac{1}{2}}\\
&\le C\varepsilon N^{-(k+1)}(\ln N)^{k+1}h_{e}^{-\frac{1}{2}}\left(\sum_{\kappa\subset\Omega/\Omega_{22}}\int_{\kappa}(\nabla\xi)^{2}\mr{d}x\mr{d}y\right)^{\frac{1}{2}}\\
&\le C\varepsilon N^{-(k+1)}(\ln N)^{k+1}(\varepsilon N^{-1}\ln N)^{-\frac{1}{2}}\left(\sum_{\kappa\subset\Omega/\Omega_{22}}\int_{\kappa}(\nabla\xi)^{2}\mr{d}x\mr{d}y\right)^{\frac{1}{2}}\\
&\le C N^{-(k+\frac{1}{2})}(\ln N)^{k+\frac{1}{2}}\left(\sum_{\kappa\subset\Omega/\Omega_{22}}\varepsilon\int_{\kappa}(\nabla\xi)^{2}\mr{d}x\mr{d}y\right)^{\frac{1}{2}}\\
&\le C N^{-(k+\frac{1}{2})}(\ln N)^{k+\frac{1}{2}}\Vert\xi\Vert_{NIPG}.
\end{aligned}
\end{equation*}
where  $h_{e}$ represents the minimum length of all edges in $\kappa\subset\Omega/\Omega_{22}$. 
Therefore, we can derive 
\begin{equation}\label{DDD-1}
|\mr{IV}|\le C N^{-(k+\frac{1}{2})}(\ln N)^{k+\frac{1}{2}}\Vert\xi\Vert_{NIPG}.
\end{equation}

Then for $\mr{V}$, using the nature of vertices-edges-element interpolation on $\overline{\Omega/\Omega_{11}}$, we just analyze it on $\overline{\Omega}_{11}$. Through \eqref{eq: penalization parameters}, \eqref{eq:QQ-5} and $\rho_{e}=1$ for $e\subset\Omega^{*}$, 
\begin{equation*}\label{convergence-1-4}
\begin{aligned}
|\sum_{e\subset\Omega^{*}}\int_{e}\rho_{e}[\eta][\xi]\mr{d}s|&\le C\left(\sum_{e\subset\Omega^{*}}\int_{e}\rho_{e}[\eta]^{2}\mr{d}s\right)^{\frac{1}{2}}\left(\sum_{e\subset\Omega^{*}}\int_{e}\rho_{e}[\xi]^{2}\mr{d}s\right)^{\frac{1}{2}}\\
&\le C\left(\sum_{e\subset\Omega^{*}}\int_{e}\rho_{e}[\eta]^{2}\mr{d}s\right)^{\frac{1}{2}}\Vert\xi\Vert_{NIPG}\\
&\le C\left(\Vert\eta\Vert^{2}_{L^{\infty}(\Omega_{11})}N^{2}N^{-1}\right)^{\frac{1}{2}}\Vert\xi\Vert_{NIPG}\\
&\le CN^{-(k+\frac{1}{2})}\Vert\xi\Vert_{NIPG}.
\end{aligned}
\end{equation*}
Furthermore, for $e\subset\partial\overline{\Omega}_{11}/\Gamma$, one also obtains
\begin{equation*}
\begin{aligned}
 |\sum_{e\subset\partial\overline{\Omega}_{11}/\Gamma}\int_{e}\rho_{e}[\eta][\xi]\mr{d}s|
&\le C\left(\sum_{e\subset\partial\overline{\Omega}_{11}/\Gamma}\int_{e}\rho_{e}[\eta]^{2}\mr{d}s\right)^{\frac{1}{2}}\Vert\xi\Vert_{NIPG}\\
&\le C\left(N\Vert\eta\Vert^{2}_{L^{\infty}(\Omega/\Omega_{22})}\right)^{\frac{1}{2}}\Vert\xi\Vert_{NIPG}\\
&\le CN^{-(k+\frac{1}{2})}(\ln N)^{k+1}\Vert\xi\Vert_{NIPG},
\end{aligned}
\end{equation*}
where $\rho_{e}=N$ for $e\subset\partial\overline{\Omega}_{11}/\Gamma$ have been used. So far, there is
\begin{equation}\label{convergence-1-4}
|\mr{V}|\le CN^{-(k+\frac{1}{2})}(\ln N)^{k+1}\Vert\xi\Vert_{NIPG}.
\end{equation}

Finally, we consider $\mr{VI}$.  Applying integration by parts, we have
\begin{equation}\label{UU-1}
\begin{aligned}
&\sum_{\kappa\in\mathcal{T}}\left(\int_{\kappa}(\textbf{b}\cdot\nabla\eta)\xi\mr{d}x\mr{d}y-\int_{\partial_{-}\kappa\cap\Gamma}(\textbf{b}\cdot\textbf{n}_{\kappa})\eta^{+}\xi^{+}\mr{d}s-\int_{\partial_{-}\kappa/\Gamma}(\textbf{b}\cdot\textbf{n}_{\kappa})\lfloor \eta\rfloor\xi^{+}\mr{d}s\right)\\
&=-\sum_{\kappa\in\mathcal{T}}\int_{\kappa}\eta(\textbf{b}\cdot\nabla\xi)\mr{d}x\mr{d}y+2\sum_{\kappa\in\mathcal{T}}\int_{\kappa}(c_{0}^{2}-c)\eta\xi\mr{d}x\mr{d}y+\sum_{\kappa\in\mathcal{T}}\int_{\partial_{+}\kappa\cap\Gamma}(\textbf{b}\cdot\textbf{n}_{\kappa})\eta^{+}\xi^{+}\mr{d}s\\
&+\sum_{\kappa\in\mathcal{T}}\left(\int_{\partial_{+}\kappa/\Gamma}(\textbf{b}\cdot\textbf{n}_{\kappa})\eta^{+}\xi^{+}\mr{d}s+\int_{\partial_{-}\kappa/\Gamma}(\textbf{b}\cdot\textbf{n}_{\kappa})\eta^{-}\xi^{+}\mr{d}s\right)\\
&=\mathcal{R}_{1}+\mathcal{R}_{2}+\mathcal{R}_{3}+\mathcal{R}_{4},
\end{aligned}
\end{equation}
where the function $c_{0}^{2}$ is defined in \eqref{eq:SPP-condition-1}. Then for $\mathcal{R}_{2}$, through the Cauchy–Schwarz inequality one has
\begin{equation}\label{DDD-3}
\begin{aligned}
|\sum_{\kappa\in\mathcal{T}}\int_{\kappa}(c_{0}^{2}-c)\eta\xi\mr{d}x\mr{d}y|&\le C\left(\sum_{\kappa\in\mathcal{T}}\Vert\eta\Vert_{L^{2}(\kappa)}^{2}\right)^{\frac{1}{2}}\Vert\xi\Vert_{NIPG}\\
&\le C\left(N^{-(k+1)}+\varepsilon^{\frac{1}{2}}N^{-(k+1)}(\ln N)^{k+1}\right)\Vert\xi\Vert_{NIPG}.
\end{aligned}
\end{equation}
For $\mathcal{R}_{3}$, using \eqref{eq:QQ-5} and \eqref{eq:QQ-6}, we have
\begin{equation}\label{DDD-4}
\begin{aligned}
&|\sum_{\kappa\in\mathcal{T}}\int_{\partial_{+}\kappa\cap\Gamma}(\textbf{b}\cdot\textbf{n}_{\kappa})\eta^{+}\xi^{+}\mr{d}s|\\
&\le C\sum_{\kappa\in\mathcal{T}}\left(\int_{\partial_{+}\kappa\cap\Gamma}(\textbf{b}\cdot\textbf{n}_{\kappa})(\eta^{+})^{2}\mr{d}s\right)^{\frac{1}{2}}\left(\int_{\partial_{+}\kappa\cap\Gamma}(\textbf{b}\cdot\textbf{n}_{\kappa})(\xi^{+})^{2}\mr{d}s\right)^{\frac{1}{2}}\\
&\le C\left(\sum_{\kappa\in\mathcal{T}}\int_{\partial_{+}\kappa\cap\Gamma}(\textbf{b}\cdot\textbf{n}_{\kappa})(\eta^{+})^{2}\mr{d}s\right)^{\frac{1}{2}}\left(\sum_{\kappa\in \mathcal{T}}\int_{\partial_{+}\kappa\cap\Gamma}(\textbf{b}\cdot\textbf{n}_{\kappa})(\xi^{+})^{2}\mr{d}s\right)^{\frac{1}{2}}\\
&\le C\left(\sum_{\kappa\in\mathcal{T}}\int_{\partial_{+}\kappa\cap\Gamma}(\textbf{b}\cdot\textbf{n}_{\kappa})(\eta^{+})^{2}\mr{d}s\right)^{\frac{1}{2}}\Vert\xi\Vert_{NIPG}\\
&\le C\left(\Vert\eta\Vert_{L^{\infty}(\Omega)}^{2}N N^{-1}\right)^{\frac{1}{2}}\Vert\xi\Vert_{NIPG}\\
&\le CN^{-(k+1)}(\ln N)^{k+1}\Vert\xi\Vert_{NIPG}.
\end{aligned}
\end{equation}
Furthermore, the last two integrals $\mathcal{R}_{4}$ in \eqref{UU-1} can be estimated as in \cite{Hou1Sch2Sui3:2002-D}, that is
\begin{equation}\label{DDD-5}
\begin{aligned}
&|\sum_{\kappa\in\mathcal{T}}\left(\int_{\partial_{+}\kappa/\Gamma}(\textbf{b}\cdot\textbf{n}_{\kappa})\eta^{+}\xi^{+}\mr{d}s+\int_{\partial_{-}\kappa/\Gamma}(\text{b}\cdot\textbf{n}_{\kappa})\eta^{-}\xi^{+}\mr{d}s\right)|\\
&\le C\left(\sum_{\kappa\in\mathcal{T}}\int_{\partial_{-}\kappa/\Gamma}(\textbf{b}\cdot\textbf{n}_{\kappa})(\eta^{-})^{2}\mr{d}s\right)^{\frac{1}{2}}\left(\sum_{\kappa\in\mathcal{T}}\int_{\partial_{-}\kappa/\Gamma}(\textbf{b}\cdot\textbf{n}_{\kappa})(\xi^{+}-\xi^{-})^{2}\mr{d}s\right)^{\frac{1}{2}}\\
&\le C\left(\Vert\eta\Vert^{2}_{L^{\infty}(\Omega)}N^{2} N^{-1}\right)^{\frac{1}{2}}\Vert\xi\Vert_{NIPG}\\
&\le CN^{-(k+\frac{1}{2})}(\ln N)^{k+1}\Vert\xi\Vert_{NIPG}.
\end{aligned}
\end{equation}

Below, let us estimate $\mathcal{R}_{1}$. On the one hand, for $\kappa\subset\Omega_{11}$, let $\overline{\textbf{b}} = (\overline{b}_{1}, \overline{b}_{2})$ denote a local piecewise constant approximation of $\textbf{b} = (b_{1}, b_{2})$. Then, we can derive
\begin{equation*}
\begin{aligned}
\sum_{\kappa\subset\Omega_{11}}\int_{\kappa}(\textbf{b}\cdot\nabla\xi)\eta\mr{d}x\mr{d}y=\sum_{\kappa\subset\Omega_{11}}\int_{\kappa}(\textbf{b}-\overline{\textbf{b}})\cdot\nabla\xi\eta\mr{d}x\mr{d}y,
\end{aligned}
\end{equation*}
which follows from the definition of the local $L^{2}$-projection operator \eqref{WW-1}. Therefore, 
\begin{equation}\label{DDD-6}
\begin{aligned}
|\sum_{\kappa\subset\Omega_{11}}\int_{\kappa}(\textbf{b}-\overline{\textbf{b}})\cdot\nabla\xi\eta\mr{d}x\mr{d}y|
&\le C\sum_{\kappa\subset{\Omega_{11}}}N^{-1}\Vert\eta\Vert_{L^{2}(\kappa)}\Vert\nabla\xi\Vert_{L^{2}(\kappa)}\\
&\le C\sum_{\kappa\subset{\Omega_{11}}}\Vert\eta\Vert_{L^{2}(\kappa)}\Vert\xi\Vert_{L^{2}(\kappa)}\\
&\le C\left(\sum_{\kappa\subset{\Omega_{11}}}\Vert\eta\Vert_{L^{2}(\kappa)}^{2}\right)^{\frac{1}{2}}\Vert\xi\Vert_{NIPG}\\
&\le CN^{-(k+1)}\Vert\xi\Vert_{NIPG},
\end{aligned}
\end{equation}
where $\textbf{b}$ is sufficiently smooth, then the Cauchy-Schwarz inequality and the local inverse inequality \citep[Theorem 3.2.6]{Cia1:2002-motified} for $\xi\in V_{N}^{k}$ have been applied. 

On the other hand, for $\kappa\subset\Omega/\Omega_{11}$, from H\"{o}lder inequalities and \eqref{eq:QQ-2},
\begin{equation}\label{DDD-7}
\begin{aligned}
&|\sum_{\kappa\subset\Omega/\Omega_{11}}\int_{\kappa}\eta(\textbf{b}\cdot\nabla\xi)\mr{d}x\mr{d}y|\\
&\le C\left(\sum_{\kappa\subset\Omega/\Omega_{11}}\Vert\eta\Vert^{2}_{L^{2}(\kappa)}\right)^{\frac{1}{2}}\left(\sum_{\kappa\subset\Omega/\Omega_{11}}\Vert\nabla\xi\Vert^{2}_{L^{2}(\kappa)}\right)^{\frac{1}{2}}\\
&\le C\left(\varepsilon N^{-2(k+1)}(\ln N)^{2(k+1)}\right)^{\frac{1}{2}}\left(\sum_{\kappa\subset\Omega/\Omega_{11}}\Vert\nabla\xi\Vert^{2}_{L^{2}(\kappa)}\right)^{\frac{1}{2}}\\
&\le CN^{-(k+1)}(\ln N)^{k+1}\Vert\xi\Vert_{NIPG}.
\end{aligned}
\end{equation}
Substituting \eqref{DDD-3}-\eqref{DDD-7} into \eqref{UU-1}, we can derive the following estimate directly, that is
\begin{equation}\label{convergence-5}
\mr{VI}\le CN^{-(k+\frac{1}{2})}(\ln N)^{k+1}\Vert\xi\Vert_{NIPG}.
\end{equation}

Last but not least, by \eqref{eq:convergence-1}, \eqref{BB-1}, \eqref{DDD-2}, \eqref{DDD-1}, \eqref{convergence-1-4} and \eqref{convergence-5}, there is
\begin{equation*}
\begin{aligned}
\Vert\xi\Vert^{2}_{NIPG}&\le\mr{I}+\mr{II}+\mr{III}+\mr{IV}+\mr{V}+\mr{VI}
\le CN^{-(k+\frac{1}{2})}(\ln N)^{k+1} \Vert\xi\Vert_{NIPG}.
\end{aligned}
\end{equation*}

Now we give the main conclusion of supercloseness.

\begin{theorem}\label{the:main result1}
Let Assumption \ref{ass:S-1} hold. Suppose that $\rho_{e}$ for $e\in\mathcal{E}$ are defined as \eqref{eq: penalization parameters} and the definition of $I_{N}u$ is presented in Section 4.1.
Then on the Shishkin mesh with $\sigma\ge k+\frac{3}{2}$ we derive
\begin{align*}
\Vert I_{N}u-u_{h}\Vert_{NIPG}+\Vert \Pi u-u_{h} \Vert_{NIPG}\le CN^{-(k+\frac{1}{2})}(\ln N)^{k+1},
\end{align*}
where $\Pi u$ is the interpolation of the exact solution to \eqref{eq:S-1}, and $u_{h}$ is the solution of \eqref{eq:SD}. 
\end{theorem}
\begin{proof}
First, by using the triangle inequality, there is
\begin{equation*}
\begin{aligned}
\Vert I_{N}u-u_{h}\Vert_{NIPG}\le \Vert I_{N}u-\Pi u\Vert_{NIPG}+\Vert \Pi u-u_{h}\Vert_{NIPG}.
\end{aligned}
\end{equation*}
Here $I_{N}u$ denotes vertices-edges-element interpolation on $\Omega$, $\Pi u$ is the interpolation defined in \eqref{eq:H-1}, while $u_{h}$ is the numerical solution of \eqref{eq:S-1}.
Note that $\Vert \Pi u-u_{h} \Vert_{NIPG}$ is known, so we just analyze the bound of $\Vert I_{N}u-\Pi u\Vert_{NIPG}$ in the following.

According to the definition of $\Pi u$, we just analyze the estimate of $\Vert I_{N}u-P_{h}u\Vert_{NIPG}$ outside the layer, where $P_{h}u$ represents the local $L^{2}$ projection. From the triangle inequality, one has
\begin{equation*}
\Vert I_{N}u-P_{h}u\Vert_{NIPG, \overline{\Omega}_{11}}^{2}\le C\Vert I_{N}u-u\Vert_{NIPG, \overline{\Omega}_{11}}^{2}+C\Vert u-P_{h}u\Vert_{NIPG, \overline{\Omega}_{11}}^{2}.
\end{equation*}
According to \eqref{post-process-1} and \eqref{post-process-2}, we obtain
\begin{equation*}
\Vert I_{N}u-P_{h}u\Vert_{NIPG, \overline{\Omega}_{11}}\le CN^{-(k+\frac{1}{2})}(\ln N)^{k+1}.
\end{equation*}
So far, the proof of this theorem has done.
\end{proof}

\section{Numerical experiment}
Below, we verify the previous theoretical conclusion by considering the following problem,
\begin{equation}\label{eq:KK-2}
\left\{
\begin{aligned}
 &-\varepsilon \Delta u+(3-x, 4-y)\nabla u+u=f,\quad x\in \Omega: = (0,1)^{2},\\
&u=0\quad\quad\text{on $\partial\Omega$},
\end{aligned}
\right.
\end{equation}
where $f(x)$ is chosen to satisfy
\begin{equation*}
u(x, y)=\sin(x)(1-e^{-2(1-x)/\varepsilon})\sin(2y)(1-e^{-3(1-y)/\varepsilon})
\end{equation*}
is the exact solution of the \eqref{eq:KK-2}.

In our numerical test, we consider $\varepsilon= 10^{-3}, 10^{-4},\cdots ,10^{-9}, k = 1, 2, 3$ and $N =8, 16, \cdots, 128$. 
Besides, for the above Shishkin mesh we take $\beta_{1}=2$, $\beta_{2}=3$, $\sigma = k + \frac{3}{2}$ and choose penalization parameters $\rho_{e}$ as \eqref{eq: penalization parameters}.

Then, we define the corresponding convergence rate by
$$p^{N}= \frac{\ln e^{N}-\ln e^{2N}}{\ln 2},$$
where
$e^{N}= \Vert I_{N}u-u_{h}\Vert_{NIPG}$ is the computation error with the mesh parameter $N$ for a particular $\varepsilon$.


\begin{table}[H]
\caption{$\Vert I_{N} u-u_h\Vert_{NIPG}$ in the case of $k=1$}
\footnotesize
\resizebox{110mm}{25mm}{
\setlength\tabcolsep{4pt}
\begin{tabular*}{\textwidth}{@{\extracolsep{\fill}} c| cccccccccccccc}
\cline{1-15}{}
            \diagbox{$N$}{$\varepsilon$} &\multicolumn{2}{c}{$10^{-3}$} &\multicolumn{2}{c}{$10^{-4}$}  &\multicolumn{2}{c}{$10^{-5}$}   
&\multicolumn{2}{c}{$10^{-6}$} &\multicolumn{2}{c}{$10^{-7}$}  &\multicolumn{2}{c}{$10^{-8}$}&\multicolumn{2}{c}{$10^{-9}$}\\

\cline{1-15}&$e^N$&$p^N$&$e^N$&$p^N$&$e^N$&$p^N$&$e^N$&$p^N$&$e^N$&$p^N$&$e^N$&$p^N$&$e^N$&$p^N$\\
\cline{1-15}

             $8$       &0.220E+0  &1.12 &0.219E+0  &1.13 &0.219E+0  &1.13 &0.219E+0  &1.13  &0.219E+0  &1.13  &0.219E+0  &1.13 &0.219E+0 &1.13\\
             $16$       &0.101E+0  &1.30  &0.999E-1  &1.33  &0.997E-1  &1.33 &0.997E-1  &1.33  &0.997E-1  &1.33  &0.997E-1  &1.33  &0.997E-1 &1.33\\
             $32$       & 0.410E-1  &1.42  & 0.397E-1  &1.46 & 0.396E-1  &1.47 & 0.396E-1  &1.47 & 0.396E-1  &1.47 & 0.396E-1  &1.47  &0.396E-1  &1.47\\
             $64$       & 0.154E-1  &1.47  & 0.144E-1  &1.55 &  0.143E-1  &1.56  & 0.143E-1  &1.56  &0.143E-1  &1.56  &0.143E-1  &1.56 &0.143E-1  &1.56\\
             $128$      &0.555E-2  &1.48 &0.493E-2  &1.60 &0.486E-2  &1.62  &0.485E-2  &1.62 &0.485E-2  &1.62 &0.485E-2  &1.62 &0.485E-2  &1.61\\
            $256$     &0.200E-2 &-- &0.163E-2 &-- &0.158E-2 &-- &0.158E-2 &-- &0.158E-2 &-- &0.158E-2 &-- &0.158E-2 &--\\
\cline{1-15}
\end{tabular*}}
\label{table:1}
\end{table}

\begin{table}[H]
\caption{$\Vert I_{N} u-u_h\Vert_{NIPG}$ in the case of $k=2$}
\footnotesize
\resizebox{110mm}{25mm}{
\setlength\tabcolsep{4pt}
\begin{tabular*}{\textwidth}{@{\extracolsep{\fill}} c| cccccccccccccc}
\cline{1-15}{}
            \diagbox{$N$}{$\varepsilon$}   &\multicolumn{2}{c}{$10^{-3}$} &\multicolumn{2}{c}{$10^{-4}$}  &\multicolumn{2}{c}{$10^{-5}$}   
&\multicolumn{2}{c}{$10^{-6}$} &\multicolumn{2}{c}{$10^{-7}$}  &\multicolumn{2}{c}{$10^{-8}$} &\multicolumn{2}{c}{$10^{-9}$}\\

\cline{1-15}&$e^N$&$p^N$&$e^N$&$p^N$&$e^N$&$p^N$&$e^N$&$p^N$&$e^N$&$p^N$&$e^N$&$p^N$&$e^N$&$p^N$\\
\cline{1-15}

             $8$       &0.752E-1  &1.47 &0.746E-1  &1.49 &0.745E-1  &1.49 &0.745E-1  &1.49 &0.745E-1  &1.49 &0.745E-1  &1.49 &0.745E-1  &1.49\\
             $16$       &0.271E-1  &1.81  & 0.265E-1  &1.85  &0.265E-1  &1.86 &0.265E-1  &1.86  &0.265E-1  &1.86  &0.265E-1  &1.86  &0.265E-1  &1.86\\
             $32$       &0.775E-2  &2.00&0.735E-2  &2.11 &0.730E-2  &2.12  & 0.730E-2  &2.12 &0.730E-2  &2.12  &0.730E-2  &2.12  &0.730E-2  &2.12\\
             $64$       &0.193E-2  &2.03  &0.171E-2  &2.25  &0.168E-2  &2.28  & 0.168E-2  &2.28  &0.168E-2  &2.29 &0.168E-2  &2.29 &0.168E-2  &2.28\\
             $128$      &0.473E-3  &1.94 &0.360E-3  &2.28  &0.346E-3  &2.37  &0.345E-3  &2.39 &0.345E-3  &2.39 &0.345E-3  &2.38 &0.345E-3  &1.85\\
             $256$     &0.124E-3 &-- &0.739E-4 &-- &0.667E-4 &-- &0.660E-4 &-- &0.659E-4 &-- &0.662E-4&--&0.961E-4 &--\\

\cline{1-15}
\end{tabular*}}
\label{table:2}
\end{table}

\begin{table}[H]
\caption{$\Vert I_{N} u-u_h\Vert_{NIPG}$ in the case of $k=3$}
\footnotesize
\resizebox{110mm}{25mm}{
\setlength\tabcolsep{4pt}
\begin{tabular*}{\textwidth}{@{\extracolsep{\fill}} c| cccccccccccccc}
\cline{1-15}{}
            \diagbox{$N$}{$\varepsilon$}   &\multicolumn{2}{c}{$10^{-3}$} &\multicolumn{2}{c}{$10^{-4}$}  &\multicolumn{2}{c}{$10^{-5}$}   
&\multicolumn{2}{c}{$10^{-6}$} &\multicolumn{2}{c}{$10^{-7}$}  &\multicolumn{2}{c}{$10^{-8}$} &\multicolumn{2}{c}{$10^{-9}$}\\

\cline{1-15}&$e^N$&$p^N$&$e^N$&$p^N$&$e^N$&$p^N$&$e^N$&$p^N$&$e^N$&$p^N$&$e^N$&$p^N$&$e^N$&$p^N$\\
\cline{1-15}

             $8$       &0.200E-1  &2.01 &0.197E-1  &2.04 &0.197E-1  &2.04 &0.197E-1  &2.04 & 0.197E-1  &2.04  &0.197E-1  &2.04 &0.197E-1 &2.04\\
             $16$       &0.498E-2 & 2.38 &0.481E-2 &2.42  &0.479E-2 &2.42 &0.479E-2 &2.42  &0.479E-2 &2.42  &0.479E-2  &2.42 &0.479E-2 & 2.42\\
             $32$       &0.959E-3  &2.54 &0.900E-3  &2.57  &0.894E-3  &2.58 &0.893E-3  & 2.58  &0.893E-3  &2.58 &0.893E-3  &2.58  &0.893E-3 &2.59\\
             $64$       & 0.165E-3  &2.57 &0.151E-3  &2.57  &0.150E-3  &2.57 &  0.149E-3  &2.57  &0.149E-3  &2.57 & 0.149E-3  &2.55 & 0.148E-3 &1.88\\
             $128$      &0.279E-4  &-- &0.255E-4  &--  &0.252E-4  &-- &0.252E-4  &-- &0.251E-4  &-- &0.254E-4  &-- &0.403E-4 &--\\

\cline{1-15}
\end{tabular*}}
\label{table:3}
\end{table}

These tables above show the sharpness of Theorem \ref{the:main result1}. Furthermore, from Table \ref{table:2} and Table \ref{table:3}, we observe that when the parameter $\varepsilon\rightarrow 0$ or the degree of polynomials $k$ and the interval number $N$ become larger, convergence degradation may occur. One possible reason is that with the change of the above conditions, the condition number of the matrix may rapidly increase, which makes solving this linear system difficult. We will focus on this issue in the near future.

\section{Acknowledgements}
\subsection{Funding}
Our research is supported by National Natural Science Foundation of China (11771257), Shandong Provincial Natural Science Foundation, China (ZR2021MA004).
\subsection{Data availability statement}
The authors confirm that the data supporting the findings of this study are available within the article and its supplementary materials.
\subsection{Conflict of interests}
The authors declare that they have no conflict of interest.

\end{document}